\documentclass[11pt,oneside]{amsart}
\usepackage{amssymb, amscd, amsmath, amsthm, latexsym, hyperref,xcolor}
\usepackage{pb-diagram, fancyhdr, graphicx, psfrag}
\usepackage[text={6.3in,8.5in},centering,letterpaper,dvips]{geometry}

\newcommand{\compactlist}{\begin{list}{$\bullet$}{\setlength{\leftmargin}{1em}}}

\newcommand{\fig}[2] { \includegraphics[scale=#1]{#2} }

\def\cs{\mathop{\#}}

\def\cfk{{\textrm{CFK}}}

\newcommand{\spinc}{\ifmmode{{\mathfrak s}}\else{${\mathfrak s}$\ }\fi}
\newcommand{\spinct}{\ifmmode{{\mathfrak t}}\else{${\mathfrak t}$\ }\fi}
\newcommand{\spincw}{\ifmmode{{\mathfrak w}}\else{${\mathfrak w}$\ }\fi}

\newtheorem{theorem}{Theorem}[section]

\newtheorem*{theorem*}{Theorem}
\newtheorem{lemma}[theorem]{Lemma}

\theoremstyle{definition}

\theoremstyle{remark}

\numberwithin{equation}{section}

\begin{document}


\title[Concordance to L-Space knots]{Concordances from connected sums of torus knots to L-space knots}
\author{Charles Livingston}
 \address{Charles Livingston: Department of Mathematics, Indiana University, Bloomington, IN 47405 }
\email{livingst@indiana.edu}
\thanks{The author was supported by  NSF-DMS-1505586.}
\maketitle
 \begin{abstract}  If a knot  is a nontrivial connected sum of positive torus knots, then it is not  concordant to an $L$--space knot.
 \end{abstract}
 
 \section{Introduction}

In a recent paper about involutive knot Floer homology,  Zemke~\cite{zemke} observed that invariants arising from   Heegaard Floer knot homology can obstruct a knot from being concordant to an $L$--space knot.   He considered examples of connected sums of torus knots, such as $T(4,5)\cs T(4,5)$ and $-T(3,4) \cs -T(4,5) \cs  T(5,6)$, and noted that in some cases alternative obstructions are available.  Here we use classical knot invariants along with the Ozsv\'ath-Szab\'o tau invariant~\cite{os2}, $\tau(K)$, to prove the following theorem.\vskip.1in
 
\noindent {\bf Theorem.} {\em Let  $\{  (p_i, q_i)  \}_{i=1, \ldots ,n}$ be a  set of pairs of relatively prime positive integers with $2\le p_i < q_i$ for all $i$ and  with $n>1$.  Then $\cs_iT(p_i,q_i)$ is not concordant to an $L$--space knot.   } 
\vskip.1in

The main idea of the proof can be illustrated with the example  $K = T(4,5)\cs T(4,5)$.  Since $\tau(K) = 12$,  if $K$ is concordant to a knot $J$, then $\tau(J) = 12$. If $J$ is  an $L$--space knot, then   the Alexander polynomial of $J$ is of degree 24.  The Alexander polynomial $\Delta_K(t) $ is the product of  cyclotomic polynomials $\phi_{10}(t)^2\phi_{20}(t)^2$.  Since the Levine-Tristram signature function~\cite{levine, tristram}  for $K$ jumps by  4 at the roots of $\phi_{20}(t)$ and $\phi_{10}(t)$, the same is true for $J$.  This implies that $ \phi_{10}(t)^2\phi _{20}(t)^2$ divides $\Delta_J(t)$.   By degree considerations, this implies $\Delta_J(t) =  \phi^2_{10}(t)\phi^2_{20} = t^{24} - 2t^{23} + \cdots +1  $.  However all coefficients of the Alexander polynomial of an $L$--space knot are $\pm 1$. 

  Krcatovich~\cite{krcatovich} proved that all $L$--space knots are  prime, so such connected sums are definitely not $L$--space knots.  A proof that such connected sums cannot be {\it concordant} to $L$--space knots appears to be inaccessible using Heegaard Floer theory alone.  The proof of the main theorem depends on a detailed analysis of the signature functions and Alexander polynomials of connected sums of torus knots.  This dependance on classical invariants seems  to be necessary, as the following example shows.  We will see that the connected sum $ T(2,3) \cs T(2,3)$ is not concordant to an $L$--space knot.  However, the torus knot $T(2,5)$ is an $L$--space knot.  The Heegaard Floer complex $\cfk^\infty( T(2,5))$  is formed from  $\cfk^\infty( T(2,3) \cs T(2,3))$ by adding an acyclic summand, and  thus known concordance invariants that arise from $\cfk^\infty(K)$   cannot alone prove that $K$ is not concordant to an $L$--space knot.\vskip.05in
 
\noindent{\it References.} Basic facts about the Alexander polynomials of torus knots are covered in textbooks on knot theory, such as \cite{burde, rolfsen}.  For facts about $L$--space knots and the necessary Heegaard Floer theory, see~\cite{os2, os3, os}.  Needed results concerning the signature function are contained in~\cite{levine, tristram}, and its behavior under cabling is described in~\cite{litherland}.  One fact about the  signature function of a knot $K$ that is  used is that the jump in the function at a number $\alpha$ is bounded above by the order of $\alpha$ as a root of the Alexander polynomial $\Delta_K(t)$; also, the jump equals that order $\alpha$ modulo two.  This follows most easily from Milnor's description~\cite{milnor} of what are now called {\it Milnor signatures}, which Matumoto~\cite{matu} proved equal the jumps in the signature function;  for   recent presentations, see~\cite{gl, kearney}.
\vskip.05in
\noindent{\it Acknowledgments.} Thanks are due to John Baldwin, Matt Hedden, Jen Hom,  David Krcatovich, Adam Levine, and Ian Zemke.
 
\section{Torus knot Alexander polynomials and signature functions}  

\subsection{Alexander polynomials of torus knots} The Alexander polynomial of the torus knot $T(p,q)$ is given by 
\begin{equation}\label{eqn:poly}  \Delta_{p,q}(t) = \frac{ (t^{pq} -1)(t-1)}{(t^p-1)(t^q-1)}.
 \end{equation}
 Roots of this are $pq$--roots of unity that are not $p$-roots of unity or $q$-roots of unity.  Letting $\phi_n(t)$ denote the $n$th cyclotomic polynomial, we have:
 
\begin{lemma}\label{thm:alex}
With notation as above, $$\Delta_{p,q}(t) =\prod \phi_{\alpha_j \beta_j}(t),$$ where the product is over the set of all pairs $(\alpha_j,\beta_j)$ for which $\alpha_j$ is a factor of $p$, $\beta_j$ is a factor of $q$, and both are greater than 1.
 \end{lemma}
 
\begin{lemma}\label{thm:alexL} Consider a set of  $n$ torus knots, $\{T(p_i, q_i)\}$, and let $d_i = (p_i-1)(q_i -1)$ be the degree of the $\Delta_{T(p_i,q_i)}(t)$.  For  $K = \cs_i T(p_i, q_i)$, the two highest degree terms of the Alexander polynomial are $t^{\sum d_i} - nt^{(\sum d_i) -1}$.
 \end{lemma}
 \begin{proof} Consider   the numerator of the product of the Alexander polynomials  when is   written in the  quotient  form described by Equation~\ref{eqn:poly}.  The leading term arises as the product of terms of  degree $\sum (p_iq_i  +1)$.  (This results from the product of the $t^{p_iq_i}$ terms times the product of the $t$ terms in the $n$ factors $(t-1)$.)  The next term, of degree one less, is obtained from a similar product, except one of the $t$ terms from a $(t-1)$ factor is not included, being replaced with $-1$ in the product.  There are $n$ such possible terms to drop.
 
The leading term of the denominator is   degree $s = \sum p_i + \sum q_i$.  The next highest degree term is of degree at most $s -2$ (which would occur if one of the $p_i$ or $q_i$ were equal to $2$). The result stated in the lemma is now easily seen, for instance by considering the long division algorithm.

\end{proof}
 
\subsection{Signature functions of torus knots}
 
For a knot $K \subset S^3$, let $\sigma_K(t)$ denote the signature of the hermitian form $(1-\omega)V_K + (1-\overline{\omega})V_K^{\sf T}$, 
where $V_K$ is a Seifert matrix and  $\omega=  e^{{2\pi i}t }$.  The associated jump function is given by 
$$J_K(t) =\frac{1}{2}\left( \lim_{s \to t^+}\sigma_K(s) -  \lim_{s \to t^-}\sigma_K(s)\right).$$
 This function is a concordance invariant of $K$. 
Figure~\ref{figsign} illustrates the graph of the signature function for $T(3,7)$ on $[0,\frac{1}{2}]$ with 
jumps at $\{\frac{1}{21}, \frac{2}{21}, \frac{4}{21}, \frac{5}{21}, \frac{8}{21}, \frac{10}{21}\} $, each of value $\pm 1$.  Notice that the jump at $\frac{1}{21}$ is negative. (The factor of $1/2$ in the definition of $J_K(t)$ is included to simplify notation. As defined, $J_K(t) = \sigma_K(t+ \epsilon) - \sigma_K(t)$ for all sufficiently small $\epsilon >0$.)
 \begin{figure}[h]
 \fig{.4}{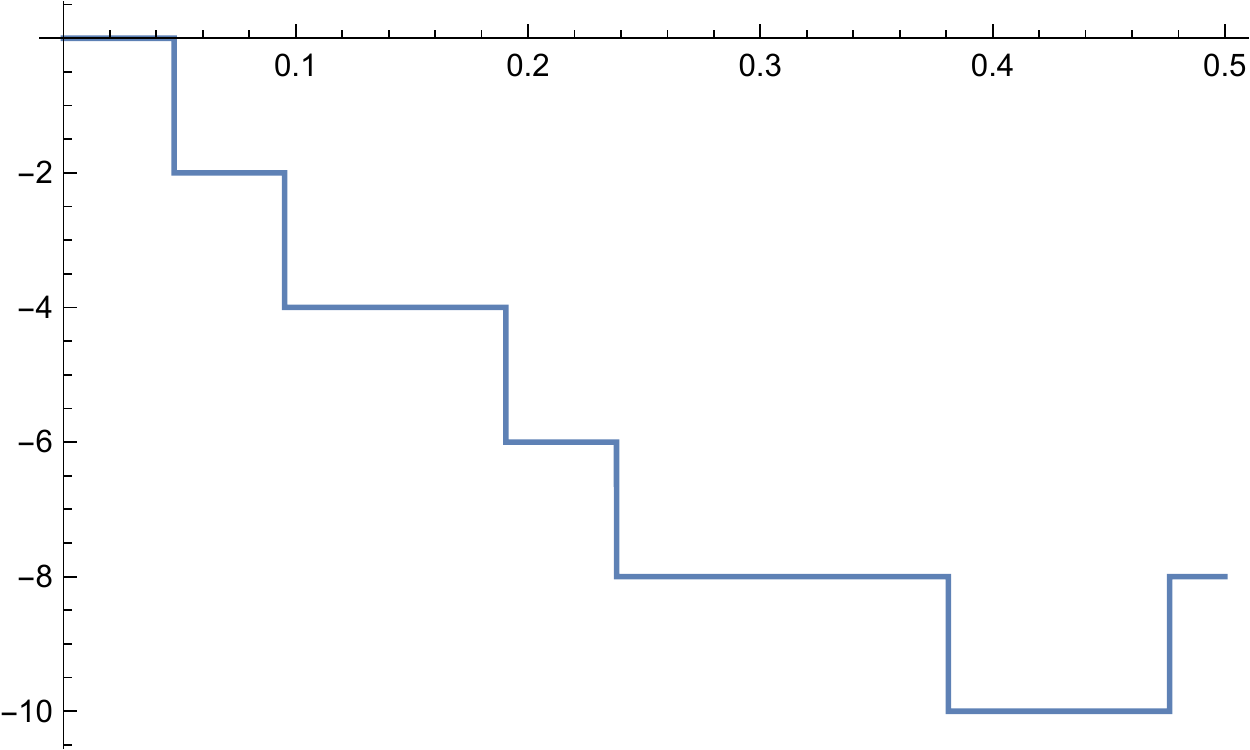}
 \caption{}\label{figsign}
 \end{figure} 

 \begin{lemma}\label{lem:main} For any positive torus knot $T(p,q)$ and for any positive integer $r$, $J_K(1/r) \le 0$.
 \end{lemma}
 
\begin{proof} We use a formula of Litherland~\cite{litherland} to study the signature function.  In this formula, we denote the  integer lattice by  $\Lambda$.  For $0 \le x \le 1$, let 
$$S_-(x) = \{ (i,j) \in \Lambda \ \big| \ 0 < i < p,\ 0 < j <q, \text{\ and\ } \frac{i}{p} +\frac{j}{q} =  x\}$$
 and
 $$S_+(x) = \{ (i,j) \in \Lambda \ \big| \ 0 < i < p,\ 0 < j <q, \text{\ and\ } \frac{i}{p} +\frac{j}{q} =  1+x\}.$$
 According to~\cite{litherland}, $J_{T(p,q)}(x)$ is given by the difference of counts: $J_{T(p,q)}(x) = \# S_-(x) - \# S_+(x)$. 

The lemma is a consequence of the observation that for any positive integer $r$,  $\# S_-(1/r)  = 0$.  To see this, we consider the equation  $$\frac{i}{p} +\frac{j}{q} =  \frac{1}{r}.$$ Multiplying by $pq$ gives
 $$iq + jp = \frac{pq}{r}.$$
 
Jumps in the signature function can occur only at roots of the Alexander polynomial.  Applying Theorem~\ref{thm:alex}, we need to consider the case of  $r = \alpha_1 \beta_1$, where $p = \alpha_1 \alpha _2$, $\alpha_1 >1$, $q = \beta_1 \beta_2$, and   $\beta_1 >1$.  Thus, our equation becomes
 $$i q + jp = \alpha_2 \beta_2.$$
 Since $\alpha_2 $ divides $p$ and $\alpha_2 \beta_2$ and is relatively prime to $q$, it must also divide $i$.  We write $i = i' \alpha_2$ and then divide by $\alpha_2$ to find
 $$ i' q + j \alpha_1 = \beta_2.$$
 Similarly, $\beta_2$ divides $q$ (and itself), 
so it divides $j \alpha_1$.  However, $\beta_2$ 
and $\alpha_1$ are relatively prime, so $j = j' \beta_2$.  Dividing yields
 $$i'\beta_1 + j' \alpha_1 = 1.$$
 
Clearly, since both summands are at least one, the sum is at least two.  Thus, there is no solution to the equation, and $S_-(1/r) = 0$, as desired.
 \end{proof}

\section{Concordances to $L$--space knots}
 We now prove our main theorem.

\begin{theorem}\label{thm:main} Let  $\{  (p_i, q_i)  \}_{i=1, \ldots ,n}$ be a  set of pairs of relatively prime positive integers with $2\le p_i < q_i$ for all $i$ and  with $n>1$.  Then $\cs_iT(p_i,q_i)$ is not concordant to an $L$--space knot.   
\end{theorem}   

\begin{proof}
Let $K = \cs_i T(p_i, q_i)$.  Then $\Delta_K(t) = \prod_j \phi_{r_j}(t)^{m_j}$ for some set of distinct $r_j$ and exponents $m_j>0$.  Notice that $m_j$ is precisely the number of pairs $(p_i, q_i)$ for which $r_j $ can be written as a product of factor of $p_i$ and a factor of $q_i$, both of which are greater than 1.   

For those $T(p_i,q_i)$ for which $\phi_{r_j}(t)$ is a factor of $\Delta_{p_i, q_i}(t)$, that factor has exponent one, and hence the jump at $1/r_j$ is either $\pm 1$.  By Lemma~\ref{lem:main}, the jump is $-1$.  It now follows that the jump in the signature function of $K$ at $t = \frac{1}{r_j}$ equals $-m_j$. 

Suppose that $J$ is concordant to $K$.  Then the jump in the signature function of $J$ at  $t = \frac{1}{r_j}$ also equals $-m_j$.   Thus, $\phi_{r_j}(t)^{m_j}$ is a factor of $\Delta_J(t)$.  It follows that  degree$(\Delta_J(t)) \ge $ degree$(\Delta_K(t))$.   

For a connected sum of torus knots, $2\tau(K) =  $ degree$( \Delta_K(t))$.  Also, for any $L$--space knot,  $2\tau(J) = $degree$( \Delta_J(t))$.  Thus, we have the inequalities 
$$ 2\tau(K) = \text{degree} ( \Delta_K(t))  \le \text{degree} ( \Delta_J(t))  = 2\tau(J).$$
But $\tau(J) = \tau(K)$, so we conclude 
$$  \Delta_K(t)  =  \Delta_J(t).$$
By Lemma~\ref{thm:alexL}, this polynomial does not have all nonzero coefficients equal to $\pm 1$, and thus it cannot be the Alexander polynomial of an $L$--space knot.
 
\end{proof}

 
\section{Generalizations}   

There are cases in which the main theorem extends to connected sums of torus knots, not all of which are positive.  The simplest example is  $K=T(2,5) \cs -T(2,3)$.  It has $\tau(K) = 1$  and Alexander polynomial $\phi_{6}(t)\phi_{10}(t)$.  Since the signature function jumps at all the 6 and 10 roots of unity, any $J$ concordant to $K$ would have its Alexander polynomial divisible by $\Delta_K(t)$, and thus would have degree greater than 2.  In particular since this degree exceeds twice the tau invaraint, $J$ could not be an $L$--space knot. 

A second example  is   $K = -T(3,4) \cs -T(4,5) \cs T(5,6)$.  For this knot, $\tau(K) = 2$, and so if it were concordant to an $L$--space knot $J$, the degree of $\Delta_J(t)$ would be four.  This is impossible, since a calculation shows that the signature function of $K$ has 32 singular points on the unit circle.

 It is a simple matter to build a large collection of examples.

The first example for which results above do not apply is $K= T(2,  9) \cs -T(2,3)$.  This knot has $\tau(K) = 3$ and Alexander polynomial $$\Delta_K(t) = \phi_{18}(t)\phi_{6}(t)^2 = (1 - t^3 + t^6)(1-t+t^2)^2.$$  The signature function jumps by 1 at each of the roots of $\phi_{18}$, and only at those roots.  Thus, the results proved above cannot rule out the possibility that $K$ is concordant to an $L$--space knot   $J$ with $\Delta_J(t) = \phi_{18}(t)$.

In fact, deeper results from Heegaard Floer knot theory apply    for $ T(2,  9) \cs -T(2,3)$.  A theorem of Hedden and Watson~\cite[Corollary 9]{hedden-watson} states the for an $L$--space knot   of genus $g$, the leading terms of the Alexander polynomial are $t^{2g} - t^{2g-1}$.  (This also follows from a recent result of Baldwin and Vela-Vick~\cite{baldwin-vela-vick} stating that if $K$ is fibered of genus $g$, then $\widehat{\rm{HFK}}(K, g-1) \ne 0$.)  However,  $ \phi_{18}(t)  = t^6 - t^3 + 1$, which is not consistent with this constraint.  Arguments along these lines yield large families of examples, but are not sufficient to give a general independence result.

Another interesting area for extending the main theorem is that of algebraic knots.  These can be described as iterated cables of torus knots of the form $T(p_1, q_1)_{(p_2,q_2), (p_3, q_3) , \ldots}$, for which all $p_i$ and $q_i$ are nonnegative and $q_{i+1} > p_iq_ip_{i+1}$ for all $i$.  (A basic reference for algebraic knots is~\cite{eisenbud-neumann}.  The algebraic properties of such knots are developed in~\cite{litherland}.)  The methods of this paper apply to show that many connected sums of algebraic knots are not concordant to $L$--space knots, but a general result has not been attained.  Thus, we end with a question.  \vskip.05in

\noindent{\bf Question:} Can a nontrivial connected sum of algebraic knots be concordant to an $L$--space knot?


\end{document}